\documentclass[11pt, a4paper, english]{amsart}
\usepackage{amsmath} 
\usepackage{amsthm} 
\usepackage{amssymb} 
\usepackage[dvipsnames]{xcolor}
\usepackage{amscd} 
\usepackage{mathtools}

\usepackage{subcaption}

\usepackage{array} 
\usepackage[cal=boondoxo,scr=euler]{mathalfa}
\usepackage[backref=page,linktocpage]{hyperref} 
\usepackage[nameinlink]{cleveref}
\usepackage{caption} %
\usepackage{graphics,graphicx} 
\usepackage{tikz,tikz-cd} 

\usetikzlibrary{graphs,graphs.standard,calc}

\usetikzlibrary{decorations.pathreplacing,}

\usepackage{enumerate} 
\usepackage{mlmodern}
\DeclareMathAlphabet{\mathsf}{OT1}{\sfdefault}{m}{n}

\newcommand{\nocontentsline}[3]{}
\newcommand{\tocless}[2]{\bgroup\let\addcontentsline=\nocontentsline#1{#2}\egroup}

\usepackage[margin=1.35in]{geometry}
\usepackage{verbatim}

\usepackage{scalerel}

\makeatletter
\def\dual#1{\expandafter\dual@aux#1\@nil}
\def\dual@aux#1/#2\@nil{\begin{tabular}{@{}c@{}}#1\\#2\end{tabular}}
\makeatother

\makeatletter
\@namedef{subjclassname@2020}{\textup{2020} Mathematics Subject Classification}
\makeatother

\DeclareMathAlphabet{\amathbb}{U}{bbold}{m}{n}

\hypersetup{
    colorlinks = true,
    linkbordercolor = {white},
    linkcolor = {BrickRed},
    anchorcolor = {black},
    citecolor = {BrickRed},
    filecolor = {cyan},
    menucolor = {BrickRed},
    runcolor = {cyan},
    urlcolor = {black}
}

\usetikzlibrary{automata}

\theoremstyle{plain}
\newtheorem{theorem}{Theorem}[section]
\newtheorem*{mastertheorem*}{Master Theorem}
\newtheorem{corollary}[theorem]{Corollary}
\newtheorem{lemma}[theorem]{Lemma}
\newtheorem{proposition}[theorem]{Proposition}

\theoremstyle{definition}
\newtheorem{definition}[theorem]{Definition}

\newtheorem{remark}[theorem]{Remark}

\crefname{theorem}{theorem}{theorems}
\Crefname{theorem}{Theorem}{Theorems}
\crefname{lemma}{lemma}{lemmas}
\Crefname{lemma}{Lemma}{Lemmas}
\crefname{proposition}{proposition}{propositions}
\Crefname{proposition}{Proposition}{Propositions}

\DeclareMathOperator{\rk}{rk}

\newcommand{\M}{\mathsf{M}}
\newcommand{\N}{\mathsf{N}}

\newcommand{\U}{\mathsf{U}}

\newcommand{\cl}{\operatorname{cl}}

\newcommand{\tr}{\operatorname{tr}}
\renewcommand{\top}{\operatorname{top}}

\title[Master Theorem for topological zeta functions]{A master theorem for\\topological zeta functions of matroids}
\author{Luis Ferroni}

\address{(L.~Ferroni)
  Universit\`a di Pisa, Pisa, Italy
}
\email{luis.ferroni@unipi.it}

\author{Lorenzo Vecchi}

\address{(L. Vecchi)
  KTH Royal Institute of Technology, Stockholm, Sweden
}

\email{lvecchi@kth.se}



\begin{document}

\begin{abstract}
We study the topological zeta function of a loopless matroid $\M$ and its Möbius transform. We provide a novel and manifest description (a ``Master Theorem'') for both functions and all of their coefficients, which can be used to give transparent solutions to several open questions and conjectures on topological zeta functions of matroids, even in greater generality than what was anticipated. As applications we solve conjectures of van der Veer (2019), Kutler (2023), and Mengesha, Miranda, and Sun (2026).
\end{abstract}

\maketitle

\section{Introduction}

The topological zeta function of a simple matroid was introduced by van der Veer
\cite{vanderVeer} as a combinatorial counterpart of the Denef--Loeser
topological zeta function of a hypersurface singularity. If $\M$ is realized by
a complex hyperplane arrangement, then $Z^{\top}_{\M}(s)$ agrees with the
topological zeta function of the product of the defining linear forms. In
general, however, it is defined purely in terms of the matroid: its construction
uses the lattice of flats, nested sets, and the combinatorics of wonderful
models. 
Somewhat unexpectedly, this function turns out to be independent of the choice of building set used. 
The function $Z^{\top}_{\M}(s)$ may be viewed as a rational-function invariant that records, in matroidal form, the Euler-characteristic data of the strata appearing in a log resolution. The following is the precise recursive definition of the topological zeta function of a matroid.

\begin{definition}
    There is a unique assignment that associates to every loopless matroid $\M$ two rational functions $Z_\M^{\top}(s)$ and $Y^{\top}_\M(s) \in \mathbb{Q}(s)$, characterized by the following properties:
        \begin{itemize}
            \item If $\rk \M = 0$, then $Z_\M^{\top}(s) = Y^{\top}_\M(s) = 1$.
            \item If $\rk \M = r > 0$ and $|E| = n$, then
            \[Z_\M^{\top}(s) = \frac{1}{ns + r} \sum_{\substack{F \in \mathcal{L}(\M)\\F \neq E}} Z^{\top}_{\M|_{F}}(s)\overline{\chi}_{\M/F}(1); \qquad Y^{\top}_\M(s) = \sum_{F \in \mathcal{L}(\M)} Z^{\top}_{\M|_{F}}(s) \mu_{F,E}.\]
        \end{itemize}
    Here, $\overline{\chi}$ denotes the reduced characteristic polynomial and $\mu$ denotes the Möbius function of the lattice of flats $\mathcal L (\M)$.
\end{definition}

The rational function $Y_{\M}^{\top}(s)\in \mathbb{Q}(s)$ is henceforth called the \emph{M\"obius transform} of the topological zeta function $Z_{\M}^{\top}(s)$; it was introduced and studied in detail by Mengesha, Miranda, and Sun in \cite{mengesha-miranda-sun}. Remarkably, Jensen, Kutler and Usatine \cite{jensen-kutler-usatine} show that the topological zeta function (and thus also its M\"obius transform) can be derived as a specialization of the so-called \emph{motivic zeta functions}. In the special case of hyperplane arrangements, these motivic zeta functions (and in particular the topological zeta functions) have also been studied in great detail by Kutler and Usatine in \cite{kutler-usatine}.

The main contribution of this paper is producing an explicit, iterative (i.e., non-recursive), formula for the topological zeta function and the corresponding M\"obius transform of an \emph{arbitrary} loopless matroid. 

\begin{mastertheorem*}\hypertarget{thm:main-coeffs-Z-and-Y}{}
Let $\M$ be a loopless matroid on a ground set $E$ with rank function $\rk_\M$. Then its topological zeta function can be computed as follows:
\[
    Z^{\top}_\M(s)
    =
    \sum_{m\geq 0}(-1)^m \Bigg{(}\sum_{\substack{\mathbf e\in E^m}}
    \prod_{j=1}^m \frac{1}{\rk_\M(\{e_1,\ldots,e_j\})}\Bigg{)}\,s^m,\]
where, for each $1\leq i\leq m$, $e_i$ denotes the $i$-th entry of the word $\mathbf{e}\in E^m$. The M\"obius transform is:
\[
    Y^{\top}_\M(s)
    =
    \sum_{m\geq 0}(-1)^m
    \Bigg{(}\sum_{\substack{\mathbf e\in E^m\\ \cl_\M(\mathbf e)=E}}
    \prod_{j=1}^m \frac{1}{\rk_\M(\{e_1,\ldots,e_j\})}\Bigg{)}\, s^m.
\]

\end{mastertheorem*}

\medskip

It was noted in the work of Jensen, Kutler, and Usatine \cite[Theorem~1.11]{jensen-kutler-usatine} that the first coefficients of the Taylor expansion of $Z^{\top}_{\M}(s)$ at the
origin are already quite rigid. They established the following formulas:
\[
    Z^{\top}_{\M}(0)=1,
    \qquad
    \left. \frac{\partial}{\partial s}Z^{\top}_{\M}(s)\right|_{s=0}=-n.
\] 

These two identities, first conjectured by van der Veer in \cite[Conjecture~1~and~2]{vanderVeer} now become straightforward to deduce from our Master Theorem. Moreover, the following result solves a further conjecture by van der Veer, which predicted a closed formula for the quadratic coefficient. 

\begin{theorem}[{\cite[Conjecture~3]{vanderVeer}}]\label{thm:coeff_s2_Z}
    If $\M$ is a simple matroid on $n$ elements, then
    \[
    [s^2]Z^{\top}_{\M}(s) = \binom{n+1}{2}.
    \]
\end{theorem}

Recently, Mengesha, Miranda, and Sun \cite{mengesha-miranda-sun} conjectured that the topological zeta function has an interesting behavior under matroid truncations. Our Master Theorem makes it immediate to settle this conjecture affirmatively.

\begin{theorem}[{\cite[Conjecture~1.5]{mengesha-miranda-sun}}]
\label{thm:main-Z}
Let $\M$ be a loopless matroid of positive rank $r$. Then, for every integer
$0\leq j<r$,
\[
    \left. \frac{\partial^j}{\partial s^j} Z^{\top}_\M(s)\right|_{s=0}
    =
    \left. \frac{\partial^j}{\partial s^j} Z^{\top}_{\tr(\M)}(s)\right|_{s=0},
\]
where $\tr(\M)$ denotes the truncation of $\M$.
\end{theorem}

This statement can also be formulated in terms of the Möbius transform of
the topological zeta function. Mengesha,
Miranda, and Sun proposed a stronger conjecture, which also follows directly from our Master Theorem.

\begin{theorem}[{\cite[Conjecture~1.6]{mengesha-miranda-sun}}]
\label{thm:main-Y}
Let $\M$ be a loopless matroid of rank $r$ with set of bases $\mathcal{B}(\M)$. Then $Y^{\top}_\M(s)$ has a zero of
order $r$ at $s=0$, and
\[
    \lim_{s\to 0}\frac{Y^{\top}_\M(s)}{s^r}
    =
    (-1)^r |\mathcal B(\M)|.
\]
\end{theorem}

Theorem~\ref{thm:main-Y} is stronger than Theorem~\ref{thm:main-Z}, as Mengesha, Miranda, and Sun explain in \cite[Section~5]{mengesha-miranda-sun}. Indeed, the
truncation formula of \cite{mengesha-miranda-sun} expresses
$Z^{\top}_{\tr(\M)}(s)-Z^{\top}_{\M}(s)$ as a rational multiple of $Y^{\top}_\M(s)$;
therefore the vanishing in Theorem~\ref{thm:main-Y} forces the two
Taylor expansions to agree up to order $r-1$.

\smallskip

A conjecture of Kutler, which we learned about in his public address at the Fields Institute in January 2023, see the recorded version in \cite{kutler-video}, predicts bounds for the coefficients of $Z^{\top}_\M(s)$ in terms of those for uniform matroids.

\begin{theorem}[Kutler's bound conjecture \cite{kutler-video}]\label{thm:dominance}
Let $\M$ be a loopless matroid of rank $r$ on $n$ elements. Then, for every $i\geq 0$,
\[
(-1)^i \frac{\partial^i}{\partial s^i}Z^{\top}_\M(s)|_{s=0} \geq (-1)^i \frac{\partial^i}{\partial s^i}Z^{\top}_{\U_{r,n}}(s)|_{s=0}. 
\]
\end{theorem}

Furthermore, the Master Theorem makes evident the following much more general property: topological zeta functions are (after a sign normalization) monotone under weak maps.

\begin{theorem}\label{thm:monotonicity}
Let $\varphi: E\to E$ induce a rank-preserving weak map $\M \to \N$ between two matroids on the same ground set $E$. Then, for every $m\geq 0$,
\[
    (-1)^m [s^m]Z^{\top}_\N(s)
    \geq
    (-1)^m [s^m]Z^{\top}_{\M}(s).
\]
\end{theorem}

\subsection*{Acknowledgments}

This project was initiated in January 2023, while we attended the 2022/23 series of seminars ``Matroids -- Combinatorics, Algebra, and Geometry'' organized at the Fields Institute (Toronto). We are thankful to Max Kutler for an inspiring talk \cite{kutler-video} containing several open problems that sparked our interest in this project, to Shiyue Li for encouraging us to write up our results, and to Robert Miranda for sharing with us a preliminary version of \cite{mengesha-miranda-sun} (their manuscript contains results that appear in a 2022 REU report led by Max Kutler). Although many of our results were found in 2023, we only make them available now, as we agreed with Miranda to release our manuscript only after \cite{mengesha-miranda-sun} was made public.

\section{The proof of the Master Theorem}

In this section we prove the Master Theorem. We will assume that the reader is familiar with the standard terminology and notation used in matroid theory. For a complete source, we refer to \cite{oxley,welsh}. Throughout the rest of the article, all matroids are assumed to be loopless. 
We will use the following standard identity.
\begin{lemma}\label{lem:sum_reduced_chi}
    Let $\M$ be a loopless matroid of rank $r > 0$. Then
    \[\sum_{F<E}\overline{\chi}_{\M/F}(1) = r.\]
\end{lemma}

\begin{proof}
    This result is obtained by specializing \cite[Lemma~2.5]{ferroni-matherne-stevens-vecchi} to $x=1$.
\end{proof}
Let $\M$ be a loopless matroid on a ground set $E$. For an integer $m\geq 0$ and a word $\mathbf e=(e_1,\ldots,e_m)\in E^m$
write
\[
    F_j(\mathbf e)=\cl_\M\{e_1,\ldots,e_j\},
\]
where $\cl_\M$ denotes the closure operator of $\M$.
For the empty word we use the convention
${\cl_\M(\varnothing)=\varnothing}$. We will be interested in the sequence of integers $\rk_\M F_j(\mathbf e)$, and we write $\rk_\M(\mathbf e)$ for $\rk_\M F_m(\mathbf e)$. 

\newtheorem*{thm:intro1}{\hyperlink{thm:main-coeffs-Z-and-Y}{Master Theorem} (first part)}
\begin{thm:intro1}
    \itshape Let $\M$ be a loopless matroid on a ground set $E$  with rank function $\rk_\M$. Then its topological zeta function can be computed as follows:
\[
    Z^{\top}_\M(s)
    =
    \sum_{m\geq 0}(-1)^m \Bigg{(}\sum_{\substack{\mathbf e\in E^m}}
    \prod_{j=1}^m \frac{1}{\rk_\M(\{e_1,\ldots,e_j\})}\Bigg{)}\,s^m,\]
where, for each $1\leq i\leq m$, $e_i$ denotes the $i$-th entry of the word $\mathbf{e}\in E^m$.
\end{thm:intro1}

\begin{proof}
We set 
\[
a_m = \sum_{\mathbf e \in E^m}\prod_{j=1}^m \frac{1}{\rk_\M F_j(\mathbf e)}
\]
and prove that
\[
[s^m]Z^{\top}_{\M}(-s) = a_m
\]
by double induction on $m$ and the rank of $\M$. If $\rk(\M) = 0$ or $m=0$, the statement holds. Now consider a matroid of rank $r > 0$ and a coefficient $m > 0$ and assume that the statement holds for all coefficients of all matroids of smaller rank and for all coefficients $a_j$, with $j < m$ for $\M$. 
By the defining recursion, if $\M$ is a matroid of rank $r$ on $n$ elements, then
\begin{equation}\label{eq:recursion Z}
(r - ns)Z^{\top}_\M(-s) = \sum_{F < E}Z^{\top}_{\M|F}(-s) \overline{\chi}_{\M/F}(1).
\end{equation}
We extract the coefficient of $s^m$ on both sides. By induction, since every flat appearing in the sum has rank smaller than $r$, the coefficient of $s^m$ in the right-hand side is 
\[
\sum_{F < E} a_m(\M|F)\overline{\chi}_{\M/F}(1) = \sum_{F < E}\sum_{\mathbf e \in F^m}\prod_{j=1}^m \frac{1}{\rk_{\M|F} F_j(\mathbf e)} \overline{\chi}_{\M/F}(1).
\]
Since for every $S \subseteq F$ we have $\rk_{\M}(S) = \rk_{\M|F}(S)$, the sum can be rewritten as 

\[
\sum_{\mathbf e \in E^m}\prod_{j=1}^m \frac{1}{\rk_{\M} F_j(\mathbf e)}\sum_{F_m(\mathbf e) \leq F < E}\overline{\chi}_{\M/F}(1) = \sum_{\mathbf e \in E^m}\prod_{j=1}^m \frac{1}{\rk_{\M} F_j(\mathbf e)} (r - \rk_\M(\mathbf e)),
\]
where the last equality follows by applying Lemma \ref{lem:sum_reduced_chi}. By \eqref{eq:recursion Z}, we obtain
\begin{equation}\label{eq:recursion2 Z}
r \ [s^m]Z^{\top}_\M(-s) = n\ [s^{m-1}]Z^{\top}_\M(-s) + \sum_{\mathbf e \in E^m} \prod_{j=1}^m \frac{1}{\rk_{\M} F_j(\mathbf e)} (r - \rk_\M(\mathbf e)).
\end{equation}
By induction on $m$, $[s^{m-1}]Z^{\top}_\M(-s) = a_{m-1}$. Moreover, we can expand
\[
na_{m-1} = \sum_{e_m \in E}\sum_{\mathbf{e'} \in E^{m-1}}\prod_{j=1}^{m-1} \frac{1}{\rk_\M F_j(\mathbf{e'})} = \sum_{\mathbf e \in E^m} \rk_\M(\mathbf e) \prod_{j=1}^m \frac{1}{\rk_{\M} F_j(\mathbf e)}.
\]
Substituting this into \eqref{eq:recursion2 Z} gives
\begin{align*}
r \ [s^m]Z^{\top}_\M(-s) &= \sum_{\mathbf e \in E^m} \rk_\M(\mathbf e) \prod_{j=1}^m \frac{1}{\rk_{\M} F_j(\mathbf e)} + \sum_{\mathbf e \in E^m} \prod_{j=1}^m \frac{1}{\rk_{\M} F_j(\mathbf e)} (r - \rk_\M(\mathbf e))\\
&= r \ \sum_{\mathbf e \in E^m} \prod_{j=1}^m \frac{1}{\rk_{\M} F_j(\mathbf e)}.
\end{align*}
Hence $[s^m] Z^{\top}_\M(-s) = a_m$, from which the statement follows.
\end{proof}

\newtheorem*{thm:intro2}{{\hyperlink{thm:main-coeffs-Z-and-Y}{Master Theorem} (second part)}}
\begin{thm:intro2}
    \itshape The M\"obius transform of the topological zeta function is
\[
    Y^{\top}_\M(s)
    =
    \sum_{m\geq 0}(-1)^m
    \Bigg{(}\sum_{\substack{\mathbf e\in E^m\\ \cl_\M(\mathbf e)=E}}
    \prod_{j=1}^m \frac{1}{\rk(\{e_1,\ldots,e_j\})}\Bigg{)}\, s^m.\]
\end{thm:intro2}
\begin{proof}
Since $Y^{\top} = Z^{\top} * \mu$,
\[
    Y^{\top}_{\M}(s) = \sum_{F\leq E}\mu(F,E) \sum_{m\geq 0}(-s)^m \sum_{\mathbf e\in F^m} \prod_{j=1}^m \frac{1}{\rk_{\M} F_j(\mathbf e)},
\]
where $\rk_{\M} F_j(\mathbf e) = \rk_{\M|F} F_j(\mathbf e)$ since the letters in $\mathbf e$ form a subset of $F$. Interchanging the sums, we get
\[
    Y^{\top}_{\M}(s) = \sum_{m\geq 0}(-s)^m  \sum_{\mathbf e\in E^m}  \prod_{j=1}^m \frac{1}{\rk_{\M} F_j(\mathbf e)} \sum_{F_m(\mathbf e) \leq F}\mu(F,E).
\]
The innermost sum is taken over the interval $[F_m(\mathbf e), E]$, and is thus equal to $1$ when $F_m(\mathbf e) = E$ and $0$ otherwise. The sum then reduces to the desired statement. 
\end{proof}

\section{Applications of the Master Theorem}

\subsection{Quadratic coefficients}
As advertised in the introduction, the following result solves a conjecture by van der Veer and in fact extends it to all loopless matroid, not necessarily simple.

\newtheorem*{thm:intro3}{Theorem~\ref{thm:coeff_s2_Z}}
\begin{thm:intro3}    Let $\M$ be a loopless matroid on $n$ elements. Then
    \[
    [s^2]Z^{\top}_{\M}(s) = \binom{n+1}{2} + c(\M),
    \]
    where $c(\M)$ is the number of circuits of size $2$ of $\M$, i.e., the number of unordered pairs of elements that are parallel to each other.
\end{thm:intro3}

\begin{proof}
Using the Master Theorem, we write
\[
[s^2]Z^{\top}_{\M}(s) = \sum_{(e_1,e_2)\in E^2}\frac{1}{\rk_\M(\{e_1\})}\frac{1}{\rk_\M(\{e_1,e_2\})}.
\]
There are $2 c(\M)$ words $(e_1,e_2)$ with $e_1\neq e_2$ and $e_1$ parallel to $e_2$, $n$ words with $e_1 = e_2$ and $2\left[\binom{n}{2} - c(\M) \right]$ words for which $\rk_\M(\{e_1,e_2\}) = 2$. In total,
\[
[s^2]Z^{\top}_{\M}(s) = 2c(\M) \cdot 1 \ +\ n \cdot 1 \ + \ 2\left[\binom{n}{2} - c(\M) \right]\cdot \frac{1}{2} = \binom{n+1}{2} + c(\M).\qedhere
\]
\end{proof}

\subsection{The conjecture on truncations}

The following solves the stronger (and thus also the weaker) conjecture posed by Mengesha, Miranda, and Sun \cite{mengesha-miranda-sun}, thereby establishing Theorem~\ref{thm:main-Z} and Theorem~\ref{thm:main-Y}.

\newtheorem*{thm:intro4}{Theorem~\ref{thm:main-Y}}
\begin{thm:intro4}
Let $\M$ be a loopless matroid of rank $r$. Then $Y^{\top}_\M(s)$ has a zero of
order $r$ at $s=0$, and
\[
    \lim_{s\to 0}\frac{Y^{\top}_\M(s)}{s^r}
    =
    (-1)^r |\mathcal B(\M)|.
\]
\end{thm:intro4}
\begin{proof}
The identity $[s^j]Y^{\top}_{\M} = 0$ for $j < r$ is clear directly from the Master Theorem, since no subset of size less than $r$ can span the whole ground set. If $m=r$, then the only contribution comes from words that correspond to ordered bases, and, each such word has weight is $\frac{1}{r!}$. Since the elements of a basis can be reordered in $r!$ ways, in the end we get a total contribution of $1$ from every basis of $\M$, hence the result. 
\end{proof}

\subsection{Kutler's bound conjecture and monotonicity}

We say that a bijection $\varphi:E\to E$ induces a \emph{weak map} $\M \to \N$ between two matroids on $E$ if, for every basis $B \in \mathcal B (\N)$, the set $\varphi^{-1}(B) = \{\varphi^{-1}(e) \mid e \in B \}$ is a basis of $\M$.

\newtheorem*{thm:intro5}{Theorem~\ref{thm:monotonicity}}
\begin{thm:intro5}
Let $\varphi: E\to E$ induce a rank-preserving weak map $\M\to\N$ between two matroids  on the same ground set $E$. Then, for every $m\geq 0$,
\[
 (-1)^m [s^m]Z^{\top}_\N(s)
 \geq
 (-1)^m [s^m]Z^{\top}_{\M}(s).
\]
\end{thm:intro5}

\begin{proof}
By the Master Theorem we write
\[
 (-1)^m [s^m]Z^{\top}_\M(s)
 =
 \sum_{(e_1,\ldots,e_m)\in E^m}
 \prod_{j=1}^m
 \frac{1}{\rk_\M\{e_1,\ldots,e_j\}}
\]
and similarly for $\N$. Since $\varphi$ induces a weak map $\M\to\N$, for every subset $S\subseteq E$ we have
\[
 \rk_\N(S)\leq \rk_\M(\varphi^{-1}(S)).
\]
Therefore, for every word $\mathbf e=(e_1,\ldots,e_m)\in E^m$ and every $j$,
\[
 \rk_\N\{e_1,\ldots,e_j\}
 \leq
 \rk_{\M}\{\varphi^{-1}( e_1),\ldots,\varphi^{-1}( e_j) \}.
\]
Taking reciprocals gives
\[
 \prod_{j=1}^m
 \frac{1}{\rk_\N\{e_1,\ldots,e_j\}}
 \geq
 \prod_{j=1}^m
 \frac{1}{\rk_{\M}\{\varphi^{-1}(e_1),\ldots,\varphi^{-1}(e_j)\}}.
\]
Summing over all words in $E^m$ proves the coefficientwise inequality.
\end{proof}

This proves Kutler's conjecture (i.e., Theorem~\ref{thm:dominance}) after noting that the identity map $\iota:E\to E$ induces a weak map $\U_{r,n}\to \M$ from the uniform matroid $\U_{r,n}$ to every matroid $\M$ of rank $r$ on $n$ elements.

\section{Final remarks}

\subsection{A matrix interpretation}

Let $\M$ be a loopless matroid. Define $A(\M) \in \mathbb Q^{|\mathcal L(\M)| \times |\mathcal{L}(\M)|}$ to be the matrix with entries
\[
A(\M)_{F,G} = \frac{|\left\{e \in E | \cl_\M(F\cup \{e\}) = G\right\}|}{\rk_\M(G)},
\]
where $F$ and $G$ are flats of $\M$ and $G\neq\varnothing$. If $G=\varnothing$ we define $A(\M)_{F,\varnothing}=0$. Note that
\[A(\M)_{F,F} = \frac{|F|}{\rk_{\M}(F)} \text{ for $F\neq\varnothing$}\qquad \text{and} \qquad A(\M)_{F,G} = \frac{|G\setminus F|}{\rk_{\M}(G)} \enspace \text{if $F$ is covered by $G$}.\] In all other cases the corresponding entry is zero.

\begin{theorem}\label{thm:matrix Z and Y}
Let $\M$ be a loopless matroid on the ground set $E$ and let $A = A(\M)$ be defined as above. Then
\[
(-1)^m[s^m]Z^{\top}_\M(s) = e_\varnothing ^{T} A^m \mathbf{1} \quad \text{and} \quad (-1)^m[s^m]Y^{\top}_\M(s) = e_\varnothing ^{T} A^m e_E,
\]
where $e_F$ is the vector with $1$ in the entry corresponding to $F$ and $0$ otherwise and $\mathbf 1$ is the vector of all ones. 
\end{theorem}
\begin{proof}
A given word $\mathbf e$ determines a weakly increasing chain of flats
\[
    \varnothing = F_0(\mathbf e) \leq F_1(\mathbf e) \leq \cdots \leq F_m(\mathbf e)
\]
with corresponding weight
\[
    \prod_{j=1}^m \frac{1}{\rk_{\M} F_j(\mathbf e)}.
\]
On the other hand, the entry $(A^m)_{\varnothing, F}$ corresponds to 
\[
    (A^m)_{\varnothing, F} = \sum_{\varnothing = F_0 \leq F_1 \leq \cdots \leq F_m = F}\prod_{i=1}^m A_{F_{i-1},F_i} = \sum_{\substack{\mathbf e \in E^m \\ F_m(\mathbf e) = F}}\prod_{j=1}^m \frac{1}{\rk_{\M} F_j(\mathbf e)}.
\]
In particular,
\[
    e_\varnothing^{T} A^m \mathbf 1 = \sum_{F \in \mathcal L (\M)} (A^m)_{\varnothing, F} = \sum_{\substack{\mathbf e \in E^m}}\prod_{j=1}^m \frac{1}{\rk_{\M} F_j(\mathbf e)} = (-1)^m[s^m]Z^{\top}_\M(s)
\]
and
\[
    e_\varnothing^{T} A^m \mathbf e_E = (A^m)_{\varnothing, E} = \sum_{\substack{\mathbf e \in E^m \\ F_m(\mathbf e) = E}}\prod_{j=1}^m \frac{1}{\rk_{\M} F_j(\mathbf e)} = (-1)^m[s^m]Y^{\top}_\M(s).\qedhere
\]
\end{proof}
\begin{corollary}\label{cor:matrix-form-compact}
    Let $\M$ be a loopless matroid on the ground set $E$ and let $A = A(\M)$ be defined as above. Then
    \[
        Z^{\top}_\M(s) = e_\varnothing^{T} (\mathbf I + sA)^{-1} \mathbf 1 \quad \text{and} \quad Y^{\top}_\M(s) = e_\varnothing^{T} (\mathbf I + sA)^{-1} e_E.
    \]
\end{corollary}
\begin{proof}
    This follows from Theorem \ref{thm:matrix Z and Y} and the expansion of $(\mathbf I + sA)^{-1}$ as
    \[
    (\mathbf I + sA)^{-1} = \mathbf I - sA + s^2 A^2  - s^3 A^3 + \cdots = \sum_{m\geq 0} (-1)^m s^m A^m.\qedhere 
    \]
\end{proof}

The matrix interpretation also lets us recover the set of candidate poles from \cite{vanderVeer}.
\begin{corollary}\label{cor:candidate poles}
    The poles of $Z^{\top}_\M(s)$ and $Y^{\top}_\M(s)$ are a subset of 
    \[
    \left\{ -\frac{\rk_\M(F)}{|F|} \mid F \in \mathcal L (\M) \setminus \{\varnothing \} \right\}.
    \]
\end{corollary}
\begin{proof}
    For a nonempty flat $F$,
    \[
    A_{F,F} = \frac{|\left\{e \in E \mid \cl_\M(F\cup \{e\}) = F \right\}|}{\rk_\M(F)} = \frac{|F|}{\rk_\M(F)},
    \]
    where the second equality follows from $F$ being a flat. Thus, by Corollary \ref{cor:matrix-form-compact}, the set of candidate poles corresponds exactly to the zeros of $\det(\mathbf I + sA)$, which is the set in the statement since $A$ is upper-triangular (after ordering the flats by inclusion).
\end{proof}

\subsection{A formula with saturated chains}
The Master Theorem gives a concrete formula for the coefficients of $Z^{\top}_\M(s)$, but it somehow hides its poles. In this subsection we rewrite that identity in terms of saturated chains in order to make the candidate poles explicit again. 
Let $\Gamma(\M)$ denote the set of all saturated chains starting at $\varnothing$ and allowed to end at an arbitrary flat:
\[
    \gamma: \varnothing=F_0\lessdot F_1\lessdot\cdots\lessdot F_\ell .
\]
For such a chain we write $\ell(\gamma) = \ell$ for its length.
We also include the chain of length $0$, consisting only of $\varnothing$.

\begin{proposition}\label{prop:chain-expansion-Z}
Let $\M$ be a loopless matroid. Then
\[
    Z^{\top}_{\M}(s)
    = \sum_{\substack{\gamma\in\Gamma(\M) }} (-s)^{\ell(\gamma)} \prod_{i=1}^{\ell(\gamma)} \frac{|F_i \setminus F_{i-1}|}{\left(\rk_\M (F_i) +|F_i|s\right)}  .
\]
\end{proposition}

\begin{proof}
We manipulate the expression given by the Master Theorem. 

A word ${\mathbf e=(e_1,\ldots,e_m)}$ determines a weakly increasing sequence of
flats
\[
    \varnothing=G_0\leq G_1\leq\cdots\leq G_m,
    \qquad
    G_j=\cl_\M\{e_1,\ldots,e_j\}.
\]
After deleting consecutive repetitions one obtains a saturated chain
\[
    \varnothing=F_0\lessdot F_1\lessdot\cdots\lessdot F_\ell .
\]
For each fixed saturated chain and for a strict jump $F_{i-1}\lessdot F_i$, the letters $e$ satisfying
\[
    \cl_\M(F_{i-1}\cup\{e\})=F_i
\]
are precisely the elements of $F_i\setminus F_{i-1}$.  Hence the total
contribution of this jump is
\[
    (-s)\frac{|F_i\setminus F_{i-1}|}{\rk_\M(F_i)}.
\]
Once the word has reached the flat $F_i$, any number of letters from $F_i$ may be
inserted without changing the current flat.  These repetitions contribute the
geometric series
\[
    \sum_{a\geq 0}
    \left(
        (-s)\frac{|F_i|}{\rk_\M(F_i)}
    \right)^a
    =
    \frac{1}{1+\frac{|F_i|}{\rk_\M F_i}s}.
\]
Multiplying the jump contributions and the repetition contributions over
$i=1,\ldots,\ell$, and then summing over all saturated chains $\gamma$, gives the
claimed formula.
\end{proof}

\begin{remark}
Van der Veer's construction \cite{vanderVeer} naturally produces the set of candidate poles for $Z^{\top}_\M(s)$ presented in Corollary \ref{cor:candidate poles}. Determining which of these candidates are actual poles is a subtler problem, since cancellations may occur in the final rational function. In the algebro-geometric setting, the Denef--Loeser topological zeta function is constructed from a log resolution, and the exceptional divisors of the resolution give candidate poles. The monodromy conjecture, proved in \cite{budur-mustata-teitler} for central hyperplane arrangements, predicts that the poles of the topological zeta function correspond to eigenvalues of the monodromy action on the cohomology of the Milnor fibres of the underlying hypersurface. Thus, in the classical setting, the passage from candidate poles to actual poles is tied to deep geometric information. In the hyperplane arrangement case, all candidate poles are proven to be eigenvalues, so the proof gives no information on which of them are actual poles. 
Van der Veer also states that it is a folklore conjecture that the eigenvalues of the monodromy at the origin of a central hyperplane arrangement are combinatorial invariants, but there is no conjectured formula.
It would be interesting to investigate whether our matrix interpretation can give a concrete solution to this cancellation problem, by checking the generalized eigenspaces of $A(\M)$.
\end{remark}

\bibliographystyle{amsalpha}
\bibliography{bibliography}
\end{document}